\documentclass[12pt]{amsart}
\usepackage[latin1]{inputenc}
\usepackage{indentfirst}
\usepackage{color}
\usepackage[mathscr]{eucal}
\usepackage{enumerate}
\usepackage{amssymb}
\usepackage[all]{xy}
\usepackage{hyperref}

\newtheorem{thm}{Theorem}[section]
\newtheorem{cor}[thm]{Corollary}
\newtheorem{lem}[thm]{Lemma}

\newtheorem{prop}[thm]{Proposition}
\theoremstyle{definition}
\newtheorem{defn}[thm]{Definition}

\theoremstyle{remark}
\newtheorem{rem}[thm]{Remark}

\numberwithin{equation}{section}
\newcommand{\norm}[1]{\left\Vert#1\right\Vert}
\newcommand{\abs}[1]{\left\vert#1\right\vert}
\newcommand{\set}[1]{\left\{#1\right\}}

\newcommand{\To}{\longrightarrow}

\newcommand{\R}{\mathbb{R}}
\newcommand{\N}{\mathbb{N}}

\DeclareMathOperator{\Hom}{Hom}

\usepackage{color}
\usepackage{enumerate}
\begin{document}


\title{Free Banach lattices generated by a lattice and projectivity}

\author[Avil\'es]{Antonio Avil\'es}
\address[Avil\'es]{Universidad de Murcia, Departamento de Matem\'{a}ticas, Campus de Espinardo 30100 Murcia, Spain
	\newline
	\href{https://orcid.org/0000-0003-0291-3113}{ORCID: \texttt{0000-0003-0291-3113} } }
\email{\texttt{avileslo@um.es}}

\author[Mart\'inez-Cervantes]{Gonzalo Mart\'inez-Cervantes}
\address[Mart\'inez-Cervantes]{Universidad de Murcia, Departamento de Matem\'{a}ticas, Campus de Espinardo 30100 Murcia, Spain
	\newline
	\href{http://orcid.org/0000-0002-5927-5215}{ORCID: \texttt{0000-0002-5927-5215} } }	

\email{gonzalo.martinez2@um.es}
%
%

\author[Rodr\'iguez Abell\'an]{Jos\'e David Rodr\'iguez Abell\'an}
\address[Rodr\'iguez Abell\'an]{Universidad de Murcia, Departamento de Matem\'{a}ticas, Campus de Espinardo 30100 Murcia, Spain 	\newline
	\href{https://orcid.org/0000-0002-2764-0070}{ORCID: \texttt{0000-0002-2764-0070} }}

\email{josedavid.rodriguez@um.es}

\author[Rueda Zoca]{Abraham Rueda Zoca}
\address[Rueda Zoca]{Universidad de Murcia, Departamento de Matem\'{a}ticas, Campus de Espinardo 30100 Murcia, Spain
	\newline
	\href{https://orcid.org/0000-0003-0718-1353}{ORCID: \texttt{0000-0003-0718-1353} }}
\email{\texttt{abraham.rueda@um.es}}
\urladdr{\url{https://arzenglish.wordpress.com}}

\thanks{}



\begin{abstract} 
In this article we deal with the free Banach lattice $FBL\langle \mathbb{L} \rangle$ generated by a lattice $\mathbb{L}$. We prove that if $FBL\langle \mathbb{L} \rangle$ is projective then $\mathbb{L}$ has a maximum and a minimum. On the other hand, we show that if $\mathbb{L}$ has maximum and minimum then $FBL\langle \mathbb{L} \rangle$ is $2$-lattice isomorphic to a $C(K)$-space. As a consequence, $FBL\langle \mathbb{L} \rangle$ is projective if and only if it is lattice isomorphic to a $C(K)$-space with $K$ being an absolute neighborhood retract.
As an application, we characterize those linearly ordered sets and Boolean algebras for which the corresponding free Banach lattice is projective.
\end{abstract}

\maketitle
 
\section{Introduction}

Free and projective objects are commonly studied in many different areas of mathematics, but in the case of Banach lattices this study, initiated by B. de Pagter and A.W. Wickstead \cite{dPW15}, is still in its early stages. 
The free Banach lattice $FBL(A)$ generated by a set $A$ is a Banach lattice together with a bounded map $\delta_A \colon  A \longrightarrow FBL(A) $ with the property that for every bounded map $T\colon  A \longrightarrow X$, with $X$ being a Banach lattice, there exists a unique Banach lattice homomorphism $\hat T\colon  FBL(A) \longrightarrow X$ such that $T=\hat T \circ \delta_A$ and $\| \hat T \|=\sup\{\|Ta\|: a\in A\}$.
De Pagter and Wickstead proved the existence and uniqueness up to Banach lattice isometries of the free Banach lattice. Nevertheless, as they highlighted, describing the norm in concrete and readily identifiable terms is not so easy.
A useful representation of $FBL(A)$ and its norm was given in \cite{ART18}. For $a \in A$, let $\delta_a\colon[-1,1]^A \longrightarrow \mathbb{R}$ be the {\it evaluation function} given by $\delta_a(x^*) = x^*(a)$ for every $x^* \in [-1,1]^A$. For $f\colon[-1,1]^A \longrightarrow \mathbb{R}$, define the norm
\begin{equation*} 
\|f\|_{FBL(A)} = \sup \set{\sum_{i = 1}^n \abs{ f(x_{i}^{\ast})} :  n \in \mathbb{N}, \, x_1^{\ast}, \ldots, x_n^{\ast} \in [-1,1]^A, \text{ }\sup_{a \in A} \sum_{i=1}^n \abs{x_i^{\ast}(a)} \leq 1 }.
\end{equation*} 
Then, the free Banach lattice generated by $A$ is (up to isometries) the Banach lattice generated by the set $\{\delta_a : a\in A\}$ inside the Banach lattice of functions in $\mathbb{R}^{[-1,1]^A}$ with finite norm $\| \cdot \|_{FBL(A)}$, endowed with the norm $\| \cdot \|_{FBL(A)}$, the pointwise order and the pointwise operations. The map $\delta_A\colon  A \longrightarrow FBL(A)$ is given by the formula $\delta_A(a)=\delta_a$ for every $a\in A$.
Since $FBL(A)$ is generated by $\{\delta_a : a\in A\}$, and every such function is continuous (with respect to the product topology), it is immediate that every function in $FBL(A) \subset \mathbb{R}^{[-1,1]^A}$ is continuous and positively homogeneous, i.e.~commutes with positive scalars.

\bigskip

Free Banach lattices were later generalized by endowing the set of generators with some extra structure. Namely, in \cite{ART18} the free Banach lattice generated by a Banach space was considered, a construction that has been recently generalized in \cite{jjttt} by restricting to certain subcategories of Banach lattices, like $p$-convex Banach lattices for $1<p<\infty$. In this paper we will focus on the free Banach lattice generated by a lattice, a construction introduced in \cite{ARA18}.

Given a lattice $\mathbb{L}$, the free Banach lattice $FBL\langle \mathbb{L} \rangle$ generated by $\mathbb{L}$ is a Banach lattice together with a lattice homomorphism with bounded rank $\delta_\mathbb{L} \colon  \mathbb{L} \longrightarrow FBL\langle \mathbb{L} \rangle$ with the property that for every norm-bounded lattice homomorphism $T\colon   \mathbb{L} \longrightarrow X$, with $X$ being a Banach lattice, there exists a unique Banach lattice homomorphism $\hat T\colon  FBL\langle \mathbb{L} \rangle \longrightarrow X$ such that $T=\hat T \circ \delta_\mathbb{L}$ and $\| \hat T \|=\|T\|:=\sup\{\|Tx\|: x\in \mathbb{L} \}$.

Notice that the lattice $\mathbb{L}$ is distributive if and only if $\delta_\mathbb{L}$ is injective \cite[Proposition 3.1]{ARA18} and therefore provides a lattice embedding of $\mathbb{L}$ 
into $FBL\langle \mathbb{L} \rangle $. Indeed, a lattice $\mathbb{L}$ can be lattice embedded into a Banach lattice if and only if $\mathbb{L}$ is distributive. 
Bearing in mind that every free Banach lattice generated by a lattice is Banach lattice isometric to a free Banach lattice generated by a distributive lattice \cite[Proposition 3.2]{ARA18}, there is no loss of generality in dealing only with distributive lattices.
Throughout this paper every lattice is assumed to be distributive.

\bigskip

We will use the representation of $FBL\langle \mathbb{L} \rangle $ given in \cite{ARA18}. For any lattice $\mathbb{L}$ we denote by $\mathbb{L}^{\ast}$ the set $$\mathbb{L}^{\ast} = \set {x^{\ast}\colon \mathbb{L} \longrightarrow [-1,1] : x^{\ast} \text{ is a lattice homomorphism}}.$$ For every $x \in \mathbb{L}$ consider the evaluation function $\delta_x \colon \mathbb{L}^{\ast} \longrightarrow \mathbb{R}$ given by $\delta_x(x^{\ast}) = x^{\ast}(x)$, and for $f \in \mathbb{R}^{\mathbb{L}^{\ast}}$, define $$ \norm{f}_{FBL\langle \mathbb{L} \rangle} = \sup \set{\sum_{i = 1}^n \abs{ f(x_{i}^{\ast})} : n \in \mathbb{N}, \text{ } x_1^{\ast}, \ldots, x_n^{\ast} \in \mathbb{L}^{\ast}, \text{ }\sup_{x \in \mathbb{L}} \sum_{i=1}^n \abs{x_i^{\ast}(x)} \leq 1 }.$$ 

Then, $FBL\langle \mathbb{L} \rangle$ is the Banach lattice generated by the set $\{\delta_x: x\in \mathbb{L}\}$ inside the Banach lattice of all functions $f\in \R ^{\mathbb{L}^*}$ with $\|f\|<\infty$, endowed with the pointwise order and the pointwise operations.
The function $\delta_{\mathbb{L}}$ is given by the formula $\delta_{\mathbb{L}}(x)=\delta_x$ for every $x \in \mathbb{L}$.

As we have pointed out, the study of projective Banach lattices was probably initiated by B.~de Pagter and A.~W.~Wickstead in \cite{dPW15} and, since its appearance, a considerable effort has been done in order to study this property in different settings (see the papers \cite{amr20projectivity,dPW15} for a characterization of projectivity in $C(K)$ spaces, \cite{amr20projectivity,amr20c0} for results in free Banach lattices generated by Banach spaces, \cite{ara19} for results in free Banach lattices generated by lattices and \cite{jjttt} for results in free Banach lattices generated by $p$-convex Banach spaces).

The aim of this paper is to continue the line of the paper \cite{ara19} and to give a characterization of when the free Banach lattice generated by a lattice $FBL\langle \mathbb L\rangle$ is projective in terms of a condition involving $\mathbb{L}$. To this end, Section \ref{SectionFBLandC(K)} is devoted to the relation between the free Banach lattice generated by a lattice and $C(K)$-spaces (this section is of its own interest, but this connection is used later to describe  projective free Banach lattices generated by a lattice). Using the classical characterization of $C(K)$-spaces in terms of $AM$-spaces with order unit, we obtain the main theorem of the section, which asserts that the free Banach lattice $FBL\langle \mathbb{L} \rangle$ generated by a lattice $\mathbb{L}$ is lattice isomorphic to a $C(K)$-space whenever the lattice $\mathbb{L}$ is bounded (see Theorem \ref{TheoremLatticeIsomorphicToC(K)}).

In Section \ref{SectionProjectivity} we study the projectivity in the setting of free Banach lattices generated by a lattice. We show that a necessary condition for $FBL\langle \mathbb{L} \rangle $ to be $\infty$-projective is that $\mathbb{L}$ is bounded (Theorem \ref{theoProjectivityboundedness}). As a consequence of these results and the characterization of projective $C(K)$-spaces obtained by de Pagter, Wickstead, and the first three authors, we conclude that $FBL\langle \mathbb{L} \rangle $ is $\infty$-projective if and only if it is lattice isomorphic to a $C(K)$-space with $K$ being an ANR (Corollary \ref{CorollaryProjectivity}).
Finally, we apply this characterization to obtain several examples of $\infty$-projective free Banach lattices in Section \ref{SectionApplications}. In particular, this is the case for the free Banach lattice generated by a bounded free lattice (Lemma \ref{LemmaProjectivityFreeLatticeWithMinimumAndMaximum}). We characterize those linearly ordered sets $\mathbb{L}$ generating an $\infty$-projective free Banach lattice as those lattices which are countable and bounded (Theorem \ref{TheoremProjectivityOfTheFBLGeneratedByALinearOrder}), and we show that the free Banach lattice generated by an infinite Boolean algebra is never $\infty$-projective (Theorem \ref{TheoremProjectivityOfTheFBLGeneratedByABooleanAlgebra}). 

Our method of proving that certain Banach lattices $FBL\langle\mathbb{L}\rangle$ are $\infty$-projective is through an isomorphism with a $C(K)$-space that is $1^+$-projective. Since the norm of the isomorphism is easily checked to be $2$, these will be in fact $2^+$-projective Banach lattices. We do not know what the optimal projectivity constants are, with one exception: it was proved in \cite{ara19} that if $\mathbb{L}$ is a finite lattice, then $FBL\langle\mathbb{L}\rangle$ is $1^+$-projective.  It was also claimed in \cite{ara19} that $FBL\langle \mathbb{L}\rangle$ cannot be $1^+$-projective for any linearly ordered set, but there is a gap in one part of the proof that affects precisely to the case when $\mathbb{L}$ is countable and bounded. No example of an $\infty$-projective Banach lattice that is not $1^+$-projective is known \cite[Section 4]{amr20projectivity}.

\section{Relation between $FBL\langle \mathbb{L} \rangle$ and $C(K)$-spaces}
\label{SectionFBLandC(K)}

This section is devoted to prove that if $\mathbb{L}$ is a lattice with minimum and maximum, then the Banach lattice $FBL\langle \mathbb{L} \rangle$ is $2$-isomorphic to a $C(K)$-space. 
First, we recall some classical concepts in the theory of Banach lattices.

\begin{defn}
	Let $X$ be a Banach lattice and $u \in X$.
	\begin{enumerate}
		
		\item We say that $X$ is an \textit{AM-space} if $\| x \vee y \| = \max \{ \|x\|, \|y\| \}$ for every $x, y \in X^+$ with $x \wedge y = 0$.
		
		\item We say that $u$ is an \textit{order unit} (or \textit{unit}) if $u \geq 0$ and for every $x \in X$ there exists $\lambda \geq 0$ such that $|x| \leq \lambda u$.
		
		\item The \textit{principal ideal generated by $u$} is the vector subspace $$X_u = \{ x \in X : \text{there exists }\lambda \geq 0 \text{ such that }|x| \leq \lambda |u| \}.$$
		
	\end{enumerate}
	
\end{defn}

\begin{rem}
	Notice that $u$ is an order unit if and only if $X = X_u$. Moreover, if $X$ is an AM-space with order unit $u$, then the formula $\|x \|_{\infty} = \inf \{ \lambda \geq 0 : |x| \leq \lambda u \}$ defines an equivalent lattice norm on $X$ (see \cite[page 94]{AA}).
\end{rem}

The following two results provide a well known useful  characterization of $C(K)$-spaces.

\begin{thm}[\mbox{\cite[Theorem 3.4]{AA}}]\label{Theorem3.4AA}
	Let $X$ be a Banach lattice and let $X_u$ be the principal ideal generated by a vector $u \in X$. Then $X_u$ with the norm  $\|x \|_{\infty} = \inf \{ \lambda \geq 0 : |x| \leq \lambda |u| \}$ is an AM-space with order unit $|u|$.
\end{thm}

\begin{thm}[\mbox{\cite[Theorem 12.28]{AliprantisBurkinshaw}}]\label{Theorem12.28AB}
	A Banach lattice $X$ is an AM-space with order unit if and only if $X$ endowed with the norm $\|\cdot\|_{\infty}$ induced by an order unit is lattice isometric to some $C(K)$-space.
\end{thm}

\begin{lem}\label{LemmaOrderUnit}
	Let $\mathbb{L}$ be a lattice with minimum $m$ and maximum $M$. Then, $|\delta_m| \vee |\delta_M|$ is an order unit in $FBL\langle \mathbb{L} \rangle$.
\end{lem}

\begin{proof}
	We are going to see that $|f| \leq \|f\| (|\delta_m| \vee |\delta_M |)$ for every $f \in FBL \langle \mathbb{L} \rangle$. Suppose, by contradiction, that there is $f \in FBL \langle \mathbb{L} \rangle$ with $|f| > \|f\| (|\delta_m| \vee |\delta_M |)$. Then, there exists $x^* \in \mathbb{L}^*$ such that $|f(x^*)| > \|f\| (|x^*(m)| \vee |x^*(M)|)$. Since functions in $FBL\langle \mathbb{L} \rangle$ are positively homogeneous, we can suppose that $x^*(m) = -1$ or $x^*(M) = 1$. But then $|f(x^*)| > \| f \|$, which is impossible.
\end{proof}

In order to give an explicit compact space $K$ for which $C(K)$ is lattice isomorphic to $FBL \langle \mathbb{L} \rangle$, we need to characterize $\Hom (FBL\langle \mathbb{L} \rangle, \R)$, i.e.~ the set of real lattice homomorphisms in $FBL \langle \mathbb{L} \rangle^*$.

\begin{prop}
	\label{PropLatticeHomom}
	Let $\mathbb{L}$ be a lattice. For $x^* \in \mathbb{L}^*$, define $\delta_{x^*} \in FBL\langle \mathbb{L} \rangle^*$ by the formula  $\delta_{x^*}(f) := f(x^*)$ for every $f\colon \mathbb{L}^* \longrightarrow \R$ in $FBL\langle \mathbb{L} \rangle$. Then, $\Hom (FBL\langle \mathbb{L} \rangle, \R) = \{ \lambda \delta_{x^*} : \lambda \geq 0 \mbox{ and } x^* \in \mathbb{L}^* \}$.
\end{prop}

\begin{proof}
	It is immediate that every element of the form  $\lambda \delta_{x^*}$ with  $\lambda \geq 0$  and $x^* \in \mathbb{L}^*$ is a real lattice homomorphism in  $FBL \langle \mathbb{L} \rangle^*$, so it is enough to prove that $\Hom (FBL\langle \mathbb{L} \rangle, \R) \subseteq \{ \lambda \delta_{x^*} : \lambda \geq 0 \mbox{ and } x^* \in \mathbb{L}^* \}$.

	Let $\varphi \in \Hom(FBL\langle \mathbb{L} \rangle,\R)$. Multiplying $\varphi$ by a small enough positive constant if necessary, we can assume without loss of generality that $\|\varphi\|\leq 1$. Set $x^* := \varphi \circ \delta_\mathbb{L}$. Notice that $x^* \in \mathbb{L}^*$. 
	Now, since $\varphi$ and $\delta_{x^*}$ are Banach lattice homomorphisms such that $\varphi \circ \delta_\mathbb{L} = \delta_{x^*} \circ \delta_\mathbb{L} =x^*$, by the uniqueness of the extension we have that $\varphi = \delta_{x^*}$.
	Therefore, $\Hom (FBL\langle \mathbb{L} \rangle, \R) = \{ \lambda \delta_{x^*} : \lambda \geq 0 \mbox{ and } x^* \in \mathbb{L}^* \}$.	
\end{proof}

We are ready to prove the following theorem.

\begin{thm}\label{TheoremLatticeIsomorphicToC(K)}
	Let $\mathbb{L}$ be a lattice with minimum $m$ and maximum $M$. Then, $FBL \langle \mathbb{L} \rangle$ is $2$-lattice isomorphic to $C(K_\mathbb{L})$, where  $K_\mathbb{L} = \{ x^* \in \mathbb{L}^* : \max\{|x^*(m)|, |x^*(M)|\} = 1 \}\subseteq [-1,1]^\mathbb{L}$ is endowed with the product topology.
\end{thm}

\begin{proof}
	Let $u = |\delta_m| \vee |\delta_M| \in FBL \langle \mathbb{L} \rangle$. By using Lemma \ref{LemmaOrderUnit} and Theorem \ref{Theorem3.4AA} we have that $FBL \langle \mathbb{L} \rangle = FBL \langle \mathbb{L} \rangle_u$ with the norm $\|f\|_{\infty} = \inf \{ \lambda \geq 0 : |f| \leq \lambda u \}$ is an AM-space with order unit $u$. Thus, by Theorem \ref{Theorem12.28AB}, $FBL \langle \mathbb{L} \rangle$ with the norm $\| \cdot \|_{\infty}$ is lattice isometric to a $C(K)$-space. 
	In the proof of Lemma \ref{LemmaOrderUnit} it has been shown that $|f|\leq \|f\| u$ for every $f\in FBL\langle \mathbb{L} \rangle$.
	Since $\|u\|\leq \|\delta_m\| + \|\delta_M \|=2$, we have that  $\|f\|_\infty \leq \|f\| \leq \|f\|_\infty \|u\| \leq 2 \|f\|_\infty$. Thus, we conclude that $FBL \langle \mathbb{L} \rangle$ is $2$-lattice isomorphic to a $C(K)$-space.
	By the proof of \cite[Theorem 12.28]{AliprantisBurkinshaw}, $K$ can be taken as 
	$$K=\{x^* \in B_{FBL \langle \mathbb{L} \rangle^*}: x^* \mbox{ is a lattice homomorphism and } \|x^*\|_\infty=x^*(u)=1 \}$$
	with the weak*-topology.
	By Proposition \ref{PropLatticeHomom}, 
	$$K=\{\delta_{x^*} \in FBL \langle \mathbb{L} \rangle ^*: x^*\in \mathbb{L}^* \mbox{ and } \max\{|x^*(m)|, |x^*(M)|\} = 1\}.$$
	Now, since the map $j\colon K \longrightarrow K_\mathbb{L}$ sending each $\delta_{x^*}$ to $x^*$ is a continuous and bijective map between compact spaces, we conclude that $K$ is weak*-homeomorphic to $K_\mathbb{L}$ with the product topology, which concludes the proof.
\end{proof}

%

\section{Projectivity of $FBL\langle \mathbb{L} \rangle$}

\label{SectionProjectivity}

In this section we study sufficient and necessary conditions for the free Banach lattice generated by a lattice to be projective. We recall first the definition of projectivity and some basic facts about the behavior of this property.

\begin{defn}\label{lambdaprojdef}
Let $\lambda>1$ be a real number. A Banach lattice $P$ is said to be \textit{$\lambda$-projective} if whenever $X$ is a Banach lattice, $J$ a closed ideal in $X$ and $Q\colon X\longrightarrow X/J$ the quotient map, then for every Banach lattice homomorphism $T\colon P\longrightarrow X/J$ there is a Banach lattice homomorphism $\hat T\colon P\longrightarrow X$ such that $T=Q\circ \hat T$ and $\Vert \hat T\Vert\leq \lambda\Vert T\Vert$. We say that $P$ is \textit{$\infty$-projective} if there is $\lambda>1$ so that $P$ is $\lambda$-projective, and that $P$ is \textit{$1^+$-projective} if it is $(1+\varepsilon)$-projective for every $\varepsilon > 0$.
\end{defn}



The following result is well-known in the general theory of categories (see \cite[Proposition 24.6.2]{Semadeni}), and tells us that $\lambda$-projectivity is transferred to complemented Banach sublattices.

\begin{lem}\label{projcomplemented}
	Let $P$ and $P'$ be Banach lattices, and let $\hat{R}\colon P\To P'$ and $\hat{i}\colon P'\To P$ be Banach lattice homomorphisms such that $\|\hat{i}\| = \|\hat{R}\|=1$ and $\hat{R}\circ\hat{i} = id_{P'}$. If $P$ is $\lambda$-projective for some $\lambda > 1$, then $P'$ is $\lambda$-projective.
\end{lem}

\begin{proof}
	In order to check the $\lambda$-projectivity of $P'$, let $Q\colon X\To X/\mathcal{J}$ and $T'\colon P'\To X/\mathcal{J}$ be as in Definition~\ref{lambdaprojdef}. Then, we can apply the $\lambda$-projectivity of $P$ considering $T=T'\circ\hat{R}$, so we get $\hat{T}\colon P\To X$ such that $Q\circ \hat{T} = T'\circ\hat{R}$ and $\|\hat{T}\|\leq \lambda \|T\|\leq \lambda \|T'\|$. The desired lift is $\hat{T}' = \hat{T} \circ \hat{i}$. On the one hand $\|\hat{T}'\| \leq \|\hat{T}\| \leq \lambda \|T'\|$, and on the other hand $Q \circ \hat{T}' = Q\circ \hat{T} \circ \hat{i} = T'\circ \hat{R} \circ \hat{i} = T'$.
\end{proof}

In a similar way, we can characterize $\infty$-projective quotients of $\infty$-projective Banach lattices by the existence of a suitable isomorphic embedding. The following proposition is a variation of \cite[Theorem 10.3]{dPW15} and \cite[Proposition 2.2]{ara19}.

\begin{prop}\label{Propquotientofprojective}
	Let $P$ be a $\lambda$-projective Banach lattice, $\mathcal{I}$ a closed ideal in $P$ and $\pi\colon P\To P/\mathcal{I}$ the quotient map. The quotient $P/\mathcal{I}$ is $\infty$-projective if and only if there exists a Banach lattice homomorphism $u\colon P/\mathcal{I}\To P$ such that $\pi\circ u = id_{P/\mathcal{I}}$. 
\end{prop}

\begin{proof}
	If $P/\mathcal{I}$ is $\infty$-projective, then the existence of $u$ is just a direct consequence of Definition~\ref{lambdaprojdef}, taking as $T=id_{P/\mathcal{I}}$ and $u=\hat{T}$. On the other hand, suppose we have the above property and we are going to show that $P/\mathcal{I}$ is $(\lambda \|u\|)$-projective.
	Take a quotient map $Q\colon X\longrightarrow X/\mathcal{J}$ and a Banach lattice homomorphism $T\colon P/\mathcal{I}\To X/\mathcal{J}$. Since $P$ is $\lambda$-projective we can find $S\colon P\To X$ with $Q\circ S = T\circ\pi$ and $\|S\|\leq \lambda\|T\circ\pi\| = \lambda \|T\|$. If we take $\hat{T}= S\circ u$, then $Q\circ \hat{T} = Q\circ S \circ u = T\circ \pi\circ u = T$ and $\|\hat{T}\| \leq \|S\|\|u\| \leq \lambda\|u\|\|T\|$ as desired. 
\end{proof}

%

Since every Banach lattice is a lattice quotient of a free Banach lattice generated by a set (\cite[Proposition 6.8]{dPW15}) and every such lattice is $1^+$-projective (\cite[Proposition 10.2]{dPW15}), this class of Banach lattices plays a fundamental role in the study of projectivity. As a consequence, the following characterization of $\infty$-projective Banach lattices of the form $FBL\langle \mathbb{L} \rangle$ is obtained. 

\begin{cor}\label{CorocaractProyectividad}
	
	Let $\mathbb{L}$ be a lattice and let $R\colon FBL(\mathbb{L}) \longrightarrow FBL \langle \mathbb{L} \rangle$ be the restriction map $R(f) = f\vert_{\mathbb{L}^*}$.	The Banach lattice $FBL\langle \mathbb{L} \rangle$ is $\infty$-projective if and only if there exists a Banach lattice homomorphism $u\colon  FBL\langle \mathbb{L} \rangle \longrightarrow FBL(\mathbb{L})$ such that $R \circ u = id_{FBL\langle \mathbb{L} \rangle}$.
\end{cor}


\bigskip

In order to study sufficient and necessary conditions for the existence of the homomorphism $u$ in Corollary \ref{CorocaractProyectividad}, we need the following auxiliary lemma, which will be used in the proof of Theorem \ref{theoProjectivityboundedness}.

\begin{lem}
\label{lemauxnormboundedlattice}
Let $A$ be a set and $\mathbb{L}$ be a norm-bounded sublattice of $FBL(A)$. Then, for every $\varepsilon>0$ there is a finite set $F \subseteq A$ such that $|f(x^*)|\leq \varepsilon$ for every $f\in \mathbb{L}$ and every $x^*\in [-1,1]^A$ satisfying $x^*|_F = 0$. 
\end{lem}
\begin{proof}
Suppose by contradiction that there is $\varepsilon>0$ such that for every finite set $F \subseteq A$ there is $x_F^*\in [-1,1]^A$ satisfying $x^*|_F = 0$ and $f_F \in  \mathbb{L}$ with $|f_F(x_F^*)|> \varepsilon$.
Recall that functions in $FBL(A)$ are continuous with respect to the product topology in $[-1,1]^A$, so indeed each $x^*_F \in [-1,1]^A$ can be taken to be finitely supported.
We construct a sequence $(f_n)_{n \in \N}$ in $\mathbb{L}$ and a sequence of increasing finite sets $(F_n)_{n\in\N}$ as follows. First, take any finite set $F_0$ and set $f_1:=f_{F_0}$ and $F_1=F_0 \cup \{x\in A: x_{F_0}^*(x) \neq 0 \}$. Then, we can define $f_n$ and $F_n$ recursively just taking $f_n:=f_{F_{n-1}}$ and $F_n=F_{n-1} \cup \{x\in A: x_{F_{n-1}}^*(x) \neq 0 \}$.

We have that $f_1 \vee \ldots \vee f_n \in \mathbb{L} $ for every $n \in \N$. We claim that this sequence is not norm-bounded.
Indeed, notice that 
$$ \sup_{x\in A} \sum_{j=1}^n |x^*_{F_j}(x)| \leq 1 ,$$
since the functions $x^*_{F_1}, \ldots, x^*_{F_n}$ have pairwise disjoint supports.
Thus, by the definition of the norm of $FBL(A)$, we have that 
$$\|f_1 \vee \ldots \vee f_n \|\geq \sum_{j=1}^n |f_j(x^*_{F_j})| \geq n\varepsilon $$ 
for every $n\in \N$, which contradicts the fact that $\mathbb{L}$ is a norm-bounded subset of $FBL(A)$.
\end{proof}

\begin{thm}
\label{theoProjectivityboundedness}
If $FBL\langle \mathbb{L} \rangle$ is $\infty$-projective then $\mathbb{L}$ has maximum and minimum.
\end{thm}
\begin{proof}
Suppose, by contradiction, that  $FBL \langle \mathbb{L} \rangle$ is $\infty$-projective and that $\mathbb{L}$ has no maximum (the proof is analogue if $\mathbb{L}$ has no minimum). By Corollary \ref{CorocaractProyectividad}, there is a Banach lattice homomorphism $u \colon FBL \langle \mathbb{L} \rangle \longrightarrow FBL(\mathbb{L})$ such that $R \circ u = id_{FBL \langle \mathbb{L} \rangle}$, where $R \colon FBL(\mathbb{L}) \longrightarrow FBL \langle \mathbb{L} \rangle$ is the restriction map, given by $R(f) = f\vert_{\mathbb{L}^*}$ for every $f \in FBL(\mathbb{L})$. If for every $x \in \mathbb{L}$ we put $\varphi_x:=u(\delta_x)$, where $\delta_x \colon \mathbb{L}^* \longrightarrow \mathbb{R}$ is the evaluation function,
then $\{\varphi_x: x\in\mathbb{L} \}=u(\delta_\mathbb{L}(\mathbb{L}))$ is a norm-bounded lattice in $FBL(\mathbb{L})$ since $\|\varphi_x\|=\|u(\delta_x)\| \leq \|u\|$ for every $x\in\mathbb{L}$. Thus, by Lemma \ref{lemauxnormboundedlattice}, there is a finite set $F \subseteq \mathbb{L}$ such that $|\varphi_x(x^*)|<\frac{1}{2}$ for every $x \in \mathbb{L}$ and every $x^* \in [-1,1]^\mathbb{L}$ satisfying $x^*|_F=0$.
Since $F$ is a finite subset in $\mathbb{L}$, there exists $z:=\sup_{x\in F} x$. By hypothesis, $\mathbb{L}$ has no maximum, so there is $z'\in \mathbb{L}$ such that $z' \not\leq z$. Recall that, since $\mathbb{L}$ is distributive, it is lattice isomorphic to a sublattice of a Boolean algebra and, therefore, we can take a lattice homomorphism $x^* \in \mathbb{L}^*$ taking only values zero and one and with $x^*(z)=0$ and $x^*(z')=1$.
Notice that, since $x^*$ only takes values zero and one, we have that $x^*(x)=0$ for every $x \in F$. Thus, $x^*|_F=0$ and therefore $|\varphi_x(x^*)|<\frac{1}{2}$ for every $x \in \mathbb{L}$. But then

$$1=x^*(z')=\delta_{z'}(x^*)=R(u(\delta_{z'}))(x^*)=u(\delta_{z'})(x^*)=\varphi_{z'}(x^*) < \frac{1}{2},$$
which is a contradiction.
\end{proof}

The class of absolute neighborhood retracts turned out to play a fundamental role in the study of projective Banach lattices.

\begin{defn}\label{DefinitionANR}
	Let $K$ be a compact Hausdorff topological space. A compact Hausdorff topological subspace $K_0$ of $K$ is a \textit{neighbourhood retract of $K$} if it is a retract of some open set of $K$, that is, there exists an open set $U$ of $K$ with $K_0 \subset U$ and a continuous map $P \colon U \longrightarrow K_0$ such that $P \vert_{K_0} = id_{K_0}$. We say that $K$ is an \textit{absolute neighbourhood retract} (ANR, for short) if it is a neighbourhood retract of every compact Hausdorff topological space containing it.
\end{defn}

It was shown by B. de Pagter and A. W. Wickstead that if $C(K)$ is $\infty$-projective, then $K$ is an ANR (see the proof of \cite[Proposition 11.7]{dPW15}). On the other hand, it was shown in \cite[Theorem 1.4]{amr20projectivity} that $C(K)$ is $1^+$-projective whenever $K$ is an ANR. 
Bearing in mind that $\infty$-projectivity is stable under renormings, the following characterization is a consequence of the aforementioned results, Theorem \ref{TheoremLatticeIsomorphicToC(K)} and Theorem \ref{theoProjectivityboundedness}.

\begin{thm}
\label{theoCharacterizationProjectivity}
$FBL \langle \mathbb{L} \rangle $ is $\infty$-projective if and only if $\mathbb{L}$ has a maximum $M$ and a minimum $m$ and $K_\mathbb{L} = \{ x^* \in \mathbb{L}^* : \max\{|x^*(m)|, |x^*(M)|\} = 1 \}\subseteq [-1,1]^\mathbb{L}$, endowed with the product topology, is an ANR.
\end{thm}

\begin{cor}
\label{CorollaryProjectivity}
$FBL \langle \mathbb{L} \rangle $ is $\infty$-projective if and only if it is lattice isomorphic to a $C(K)$-space with $K$ being an ANR.
\end{cor}

\section{Examples and applications}
\label{SectionApplications}

In this section we study the projectivity of $FBL\langle \mathbb{L} \rangle$ for some particular classes of lattices such as linearly ordered sets and Boolean algebras. First, we present an important class of lattices $\mathbb{L}$ for which $FBL\langle \mathbb{L} \rangle$ is projective, which is the class of freely generated lattices together with a maximum and a minimum. Given a set $A$, we denote by $bfree(A)$ the lattice freely generated by $A$, together with a maximum $M$ and a minimum $m$. The lattice $bfree(A)$ can be identified with the sublattice of $\{0,1\}^{\{0,1\}^A}$ generated by the functions $\pi_a \in \{0,1\}^{\{0,1\}^A}$ given by the formula $\pi_a(f)=f(a)$ for every $f \in \{0,1\}^A$ and every $a \in A$, together with the constant functions zero and one.

\begin{lem}\label{LemmaProjectivityFreeLatticeWithMinimumAndMaximum}
	For any set $A$, $FBL \langle bfree(A) \rangle$ is $\infty$-projective.
%
\end{lem}

\begin{proof}
	By Theorem \ref{theoCharacterizationProjectivity}, it is enough to show that the compact space
	 $K = \{ x^* \in bfree(A)^* : \max\{|x^*(m)|, |x^*(M)|\} = 1 \}$ is an ANR.
It is enough to prove that $K$ is homeomorphic to a neighborhood retract of $[-1,1]^{\Gamma}$, with $\Gamma = A \cup \{m, M\}$, cf. \cite{Rodabaugh}.
	
	First, observe that $K$ is homeomorphic to $$K' = \{x^* \in [-1,1]^\Gamma : x^*(m) \leq x^*(a) \leq x^*(M)\text{ } \forall a \in A, \max\{|x^*(m)|, |x^*(M)|\} = 1 \}.$$
	
	Let us see that $K'$ is a retract of $[-1,1]^{\Gamma}$. For it, let $r \colon [-1,1]^2 \longrightarrow \{(x,y) \in [-1,1]^2 : x = -1 \text{ or }y = 1\}$ be a continuous retraction from the square $[-1,1]^2$ onto the space $\{(x,y) \in [-1,1]^2 : x = -1 \text{ or }y = 1\}$. 
	
	If we write $r(z) = (r_1(z), r_2(z))$ for every $z \in [-1,1]^2$, the map $R \colon [-1,1]^{\Gamma} \longrightarrow K'$ given by $R(f)(m) = r_1(f(m), f(M))$, $R(f)(M) = r_2(f(m), f(M))$, and $R(f)(a) = (f(a) \wedge R(f)(M)) \vee R(f)(m)$ for every $a \in A$, and for every $f \in [-1,1]^{\Gamma}$, is the desired retraction. 
\end{proof}

\bigskip

A simple way to obtain more examples of projective Banach lattices of the form $FBL\langle \mathbb{L} \rangle $ is by taking complemented sublattices. We say that a sublattice $\mathbb{L} \subset \mathbb{M}$ is \textit{complemented} in $\mathbb{M}$ if there exists a lattice homomorphism $R \colon \mathbb{M} \longrightarrow \mathbb{L}$ whose restriction to $\mathbb{L}$ is the identity.

\begin{prop}\label{PropLatticeComplemented}
	Let $\mathbb{L} \subset \mathbb{M}$ be two lattices. If $\mathbb{L}$ is complemented in $\mathbb{M}$, then $FBL\langle \mathbb{L} \rangle$ is a 1-complemented Banach sublattice of $FBL\langle \mathbb{M} \rangle$.
\end{prop}

\begin{proof}
	Let $i \colon \mathbb{L} \longrightarrow \mathbb{M}$ be the inclusion lattice homomorphism and let $R \colon \mathbb{M} \longrightarrow \mathbb{L}$ be the retraction. Using the universal property of the free Banach lattice generated by a lattice, we can find Banach lattice homomorphisms $\hat{i} \colon FBL \langle \mathbb{L} \rangle \longrightarrow FBL \langle \mathbb{M} \rangle$ and $\hat{R} \colon FBL \langle \mathbb{M} \rangle \longrightarrow FBL \langle \mathbb{L} \rangle$ such that $\| \hat{i} \| = \| \hat{R} \| = 1$ and $\hat{R} \circ \hat{i} = id_{FBL\langle \mathbb{L} \rangle}$, so we are done. 
\end{proof}

As a corollary of the last proposition and Lemma \ref{projcomplemented}, we get the following.

\begin{cor}\label{CorProjFreeLatticeComplemented}
	Let $\mathbb{L} \subset \mathbb{M}$ be two lattices. If $\mathbb{L}$ is complemented in $\mathbb{M}$ and $FBL\langle \mathbb{M} \rangle$ is $\lambda$-projective for some $\lambda > 1$, then $FBL\langle \mathbb{L} \rangle$ is $\lambda$-projective. 
\end{cor}

\bigskip


\begin{thm}\label{TheoremProjectivityOfTheFBLGeneratedByALinearOrder}
	Let $\mathbb{L}$ be a linearly ordered set. Then, $FBL\langle \mathbb{L} \rangle$ is $\infty$-projective if and only if $\mathbb{L}$ is countable and has a minimum and a maximum.
\end{thm}

\begin{proof}
	Suppose that $\mathbb{L}$ is countable and has a minimum $m$ and a maximum $M$. We are going to prove that $\mathbb{L}$ is a complemented sublattice of $\mathbb{M}:=bfree(\{e_n:n\in \N\})$. Once this is done, by Corollary \ref{CorProjFreeLatticeComplemented}, the proof of one implication be finished.
	
	Let us write $\mathbb{L} = \{a_n : n \in \mathbb{N}\} \cup \{m,M\}$. Let us define, by induction, an increasing mapping $i \colon \mathbb{L} \longrightarrow \mathbb{M}$. Set $i(m) = m$ and $i(M) = M$. Define $i(a_1) = e_1$, which is clearly increasing in $\{m,a_1,M\}$. Now assume that $i$ has been defined in $\{m,a_1,\ldots, a_{n-1},M\}$ and define $$i(a_n) = \left(\bigwedge \{i(a_j) : j < n, a_j > a_n\}\right) \wedge \left(\bigvee \{i(a_j) : j < n, a_j < a_n\} \vee e_n\right).$$ To finish the inductive step we have to prove that $i$ is increasing in $\{m,a_1,\ldots a_n,M\}$. It is obvious that it is enough to compare $i(a_n)$ with $i(a_k)$ for $k<n$. Now we have two possibilities:
	\begin{itemize}
		\item If $a_n<a_k$ then $i(a_n) < \bigwedge \{i(a_j) : j < n, a_j > a_n\}\leq i(a_k)$.
		\item If $a_k<a_n$ then $i(a_k)\leq\bigvee \{i(a_j) : j < n, a_j < a_n\}\leq \bigvee \{i(a_j) : j < n, a_j < a_n\}\vee e_n$. On the other hand, $i(a_k)\leq \bigwedge \{i(a_j) : j < n, a_j > a_n\}$ because, by induction hypothesis, $i$ is increasing on $\{m,a_1,\ldots, a_{n-1},M\}$. Consequently $i(a_k)\leq (\bigwedge \{i(a_j) : j < n, a_j > a_n\}) \wedge (\bigvee \{i(a_j) : j < n, a_j < a_n\} \vee e_n)=i(a_n)$, as desired.
	\end{itemize}
	
	This proves that $i$ is increasing. On the other hand, since the generators $\{e_n, n\in \N\}$ are free, $i$ is injective and, therefore, it is a lattice homomorphism from $\mathbb{L}$ into $\mathbb{M}$ since $\mathbb L$ is linearly ordered.
	
	Moreover, there is a unique lattice homomorphism $R \colon \mathbb{M} \longrightarrow \mathbb{L}$ satisfying that $R(e_n) = a_n$ for every $n \in \mathbb{N}$, $R(m) = m$, and $R(M) = M$. A standard induction argument shows that $R \circ i = id_{\mathbb{L}}$, so we are done.

	For the converse, by Theorem \ref{theoCharacterizationProjectivity}, it is enough to show that if $\mathbb{L}$ is uncountable then $FBL\langle \mathbb{L}\rangle$ is not $\infty$-projective. If $FBL \langle \mathbb{L} \rangle$ were $\infty$-projective, by Corollary \ref{CorocaractProyectividad}, there would exist a Banach lattice homomorphism $u \colon FBL \langle \mathbb{L} \rangle \longrightarrow FBL(\mathbb{L})$ such that $R \circ u = id_{FBL \langle \mathbb{L} \rangle}$, where $R \colon FBL(\mathbb{L}) \longrightarrow FBL \langle \mathbb{L} \rangle$ is the restriction map, given by $R(f) = f\vert_{\mathbb{L}^*}$ for every $f \in FBL(\mathbb{L})$. If for every $x \in \mathbb{L}$ we put $\varphi_x:=u(\delta_x)$, where $\delta_x \colon \mathbb{L}^* \longrightarrow \mathbb{R}$ is the evaluation function, we would have that $\varphi_x\vert_{\mathbb{L}^*} = \delta_x$, and $\varphi_x \leq \varphi_y$ whenever $x \leq y$.
	
	For every $x \in \mathbb{L}$, since $\varphi_x \colon [-1,1]^{\mathbb{L}} \longrightarrow \mathbb{R}$ is continuous (in the product topology), there exist a finite set $F_x \subset \mathbb{L}$ and a positive real number $\delta_x > 0$ such that whenever $p, q \in [-1,1]^{\mathbb{L}}$ satisfy that $|p(t)-q(t)| < \delta_x$ for every $t \in F_x$, then $|\varphi_x(p) - \varphi_x(q)| < \frac{1}{4}$.
	
	Since $\mathbb{L}$ is not countable, by \cite[Theorem 11.3.13]{jech} there exists an uncountable set $X \subset \mathbb{L}$ such that $\{F_x : x \in X \}$ is a $\Delta$-system, i.e.~there exists a finite set $S \subset \mathbb{L}$ such that $F_x \cap F_y = S$ for every $x, y \in X$ with $x \neq y$.
	Set $S:=\{b_1<b_2<\ldots<b_N\}$.
	Since $X$ is infinite, at least one of the 
	open intervals $(m,b_1),~(b_1,b_2), \ldots,~(b_{N-1},b_N),~(b_N, M)$ contains infinitely many elements of $X$. In particular,
	there exist $x, y \in X$, with $x < y$, such that the closed interval $[x,y]$ satisfies $[x,y] \cap S = \emptyset$. Now, let $\psi' \colon F_x \cup F_y \longrightarrow [-1,1]$ be the map defined in the following way: For $z \in F_x$, $\psi'(z) = 0$ if $z < x$, $\psi'(z) = \frac{3}{4}$ if $z = x$, and $\psi'(z) = 1$ if $z > x$. For $z \in F_y$, $\psi'(z) = 0$ if $z < y$, $\psi'(z) = \frac{1}{4}$ if $z = y$, and $\psi'(z) = 1$ if $z >y$. Note that $\psi'$ is well-defined since $F_x \cap F_y =S$ and $[x,y] \cap S = \emptyset$. Let $\psi \colon \mathbb{L} \longrightarrow [-1,1]$ be any extension of $\psi'$ to $\mathbb{L}$. It is clear that $\psi\vert_{F_x}$ and $\psi\vert_{F_y}$ are order-preserving, so there exist $\psi_1, \psi_2 \in \mathbb{L}^*$ such that $\psi\vert_{F_x} = \psi_1$ and $\psi\vert_{F_y} = \psi_2$.
	
	Now, by the very definition of $F_x$, $|\varphi_x(\psi) - \varphi_x(\psi_1)| < \frac{1}{4}$. Since $\varphi_x(\psi_1) = \psi_1(x) = \frac{3}{4}$, we have that $\varphi_x(\psi) > \frac{2}{4}= \frac{1}{2}$. Analogously, by the very definition of $F_y$, $|\varphi_y(\psi) - \varphi_y(\psi_2)| < \frac{1}{4}$. Since $\varphi_y(\psi_2) = \psi_2(y) = \frac{1}{4}$, we have that $\varphi_y(\psi) < \frac{2}{4}=\frac{1}{2}$, which contradicts that $\varphi_x \leq \varphi_y$.
\end{proof}
%
%

\bigskip

We finish this section by dealing with another well-known class of lattices, the class of Boolean algebras. Recall that in every Boolean algebra $\mathbb{B}$ there is always a minimum $m$ and a maximum $M$ and, in addition, for every $x \in \mathbb{B}$ there is an element $x^c \in \mathbb{B}$ such that $x \vee x^c = M$ and $x \wedge x^c= m$. This implies that if $x^* \in \mathbb{B}^*$, then $x^*(x) \vee x^*(x^c) = x^*(M)$ and $x^*(x) \wedge x^*(x^c)= x^*(m)$, which in turn implies that $x^*(x) \in \{x^*(m),x^*(M)\}$ for every $x \in \mathbb{B}$. Thus, every real lattice homomorphism on a Boolean algebra takes at most two values.
This property, together with the following lemma, are the key ingredients of the proof of Theorem \ref{TheoremProjectivityOfTheFBLGeneratedByABooleanAlgebra}.

\begin{lem}{\cite[Corollary 16]{Rodabaugh}}
	\label{LemmaANR}
	Let $K_0 \subset K$ be two compact Hausdorff topological spaces. If $K$ is an ANR and $K_0$ is a neighbourhood retract of $K$, then $K_0$ is an ANR.
\end{lem}

\begin{thm}\label{TheoremProjectivityOfTheFBLGeneratedByABooleanAlgebra}
	Let $\mathbb{B}$ be an infinite Boolean algebra. Then, $FBL\langle \mathbb{B} \rangle$ is not $\infty$-projective.
\end{thm}

\begin{proof}
	By Theorem \ref{theoCharacterizationProjectivity},  it is enough to show that $K_\mathbb{B} = \{ x^* \in \mathbb{B}^* : \max\{|x^*(m)|,|x^*(M)|\} = 1 \}$ is not ANR, where $m$ denotes the minimum and $M$ the maximum of $\mathbb{B}$.
	Let $K_0 = \{ x^* \in K_\mathbb{B} : x^*(m) = -1 \text{ and }x^*(M) = 1 \}$, and $U = \{ x^* \in K_\mathbb{B} : x^*(m) < 0 \text{ and }x^*(M) > 0 \}$. Note that $U$ is an open set in $K_\mathbb{B}$ which contains $K_0$. Moreover, notice that $x^*(x) \in [-1,0)\bigcup(0,1]$ for every $x^* \in U$ and every $x \in \mathbb{B}$, since homomorphisms in $\mathbb{B}^*$ takes at most two values.

	Suppose, by contradiction, that $K_\mathbb{B}$ is an ANR. The map $P \colon U \longrightarrow K_0$ given by $P(x^*)(x) = 1$ if $x^*(x) > 0$ and $P(x^*)(x) = -1$ otherwise, for every $x^* \in U$, $x \in \mathbb{B}$, is continuous and satisfies that $P \vert_{K_0} = id_{K_0}$, so $K_0$ is a retract of $U \subset K_\mathbb{B}$, and thus, $K_0$ is a neighborhood retract of $K_\mathbb{B}$. By Lemma \ref{LemmaANR}, $K_0$ is an infinite totally disconnected ANR, which contradicts \cite[Proposition 11.9]{dPW15}.
\end{proof}

\section*{Acknowledgements}
Research partially supported by Fundaci\'{o}n S\'{e}neca [20797/PI/18] and by project MTM2017-86182-P (Government of Spain, AEI/FEDER, EU). The research of G. Mart\'inez-Cervantes
was co-financed by the European Social Fund and the Youth European Initiative under Fundaci\'on S\'eneca
[21319/PDGI/19]. J. D. Rodr\'iguez Abell\'an was also supported by project 20262/FPI/17. Fundaci\'on S\'eneca. Regi\'on de Murcia (Spain). The research of A. Rueda Zoca was also supported by Juan de la Cierva-Formaci\'on fellowship FJC2019-039973, by MICINN (Spain) Grant PGC2018-093794-B-I00 (MCIU, AEI, FEDER, UE), by Junta de Andaluc\'ia Grant A-FQM-484-UGR18 and by Junta de Andaluc\'ia Grant FQM-0185.

\end{document}